\documentclass[11pt]{amsart}
\usepackage{latexsym}
\usepackage{amsthm}
\usepackage{amssymb}
\usepackage[utf8]{inputenc}
\usepackage[T1]{fontenc}
\usepackage[all]{xy}
\usepackage{amsfonts}
\usepackage{amsmath}
\usepackage{graphicx}
\usepackage{pstricks}
\usepackage{url}
\usepackage{hyperref}
\usepackage[a4paper]{geometry}
\author{A. Nibirantiza} 
\address[Aboubacar Nibirantiza]{University of Burundi, Institute of Applied Pedagogy, Department of mathematics, B.P 2523, Bujumbura-Burundi}
\email{aboubacar.nibirantiza[at].ub.edu.bi}
\date{\today} 
\title[Heisenberg order of differential operators on the superspaces...]{Heisenberg order of differential operators on the superspaces $\R^{2l+1|n}$ }
\newtheorem{theorem}{Theorem}[section]

\newtheorem{df}[theorem]{Definition}

\newtheorem{rmk}[theorem]{Remark}

\newtheorem{prop}[theorem]{Proposition}

\newcommand{\R}{\mathbb{R}}

\newcommand{\N}{\mathbb{N}}

\newcommand{\id}{\mathrm{id}}

\newcommand{\D}{\mathcal{D}}

\newcommand{\Id}{\mathrm{Id}}
\newcommand{\spo}{\mathfrak{spo}}

\begin{document}
\keywords{Contact structure on superspaces, supergeometry, differential operators.}
\maketitle
\begin{abstract}
We study in this paper, the existence of tree types of filtrations of the space $\mathcal{D}_{\lambda\mu}(\R^{2l+1|n})$ of differential operators on the superspaces $\R^{2l+1|n}$ endowed with the standard contact structure $\alpha$. On this space $\mathcal{D}_{\lambda\mu}(\R^{2l+1|n})$, we have the first filtration called canonical and because of the existence of the contact structure on superspaces $\R^{2l+1|n}$ we obtain the second filtration on the space $\mathcal{D}_{\lambda\mu}(\R^{2l+1|n})$ called filtration of Heisenberg and thus the space $\mathcal{D}_{\lambda\mu}(\R^{2l+1|n})$ is therefore denoted by $\mathcal{H}_{\lambda\mu}(\R^{2l+1|n})$. We have also a new filtration induced on $\mathcal{D}_{\lambda\mu}(\R^{2l+1|n})$ by the two filtrations and it calls bifiltration. Explicitly, the space $\mathcal{D}_{\lambda\mu}(\R^{2l+1|n})$ of  differential operators is filtered canonically by the order of its differential operators and the order is $k\in \N$. When it is filtered by order of Heisenberg, the order of any differential operator is equal to $d\in\frac{1}{2}\N$.
This study is the generalization, in super case, of the model studied by C.H.Conley and V.Ovsienko in \cite{CoOv12}.
Finally, we show that the $\spo(2l+2|n)$-module structure on the space  $\mathcal{D}_{\lambda\mu}(\R^{2l+1|n})$ of differential operators is induced on the  space $\mathcal{H}_{\lambda\mu}(\R^{2l+1|n})$ and therefore on the associated space $\mathcal{S}_\delta(\R^{2l+1|n})$ of normal symbols, on the space $\mathcal{P}_\delta(\R^{2l+1|n})$ of symbols of Heisenberg and on the space of fine symbol $\Sigma_\delta(\R^{2l+1|n})$. 
\end{abstract}

\section{Introduction}
This paper is based on the concepts of contact supergeometry. We start by the standard contact structure $\alpha$ on the supermanifold $\R^{2l+1|n}$. This paper is in some way a generalization of those formulas known in classical geometry, as done by C.H. Conley and V. Ovsienko in \cite{CoOv12}. \\
In the second time we describe the concept of densities on the standard supermanifold $\R^{2l+1|n}$ in the classical way. We define two types of densities on $\R^{2l+1|n}$ and we show that if $d$ is a superdimension which different to $-1$, there is an isomorphism between the space of tensor densities $\mathrm{Ber}_{\frac{\lambda}{d+1}}(\R^{2l+1|n})$ and the space of contact densities $\mathcal{F}_\lambda(\R^{2l+1|n})$ where $\lambda\in\R$. These spaces $\mathrm{Ber}_{\frac{\lambda}{d+1}}(\R^{2l+1|n})$ and $\mathcal{F}_\lambda(\R^{2l+1|n})$ are isomorphic as $\mathcal{K}(2l+1|n)$-modules where $\mathcal{K}(2l+1|n)$ is the Lie superalgebra of contact vector fields on the contact supermanifold $\R^{2l+1|n}$.

The definition of differential operators which act between the $\lambda$-densities and $\mu$-densities, where $\lambda,\mu\in\R$, on the supermanifold $\R^{2l+1|n}$ proves the existence on the space $\mathcal{D}_{\lambda\mu}(\R^{2l+1|n})$ of two types of filtrations, i.e. the canonical filtration and the filtration of Heisenberg. These two filtrations induce another filtration called bifiltration. In this way, we generalize in super case, the model described in the even case by C.H. Conley and V. Ovsienko in \cite{CoOv12}.
We define the spaces $\mathcal{S}_\delta(\R^{2l+1|n})$ of symbol associated to the canonical filtration and the space of symbol of Heisenberg $\mathcal{P}_\delta(\R^{2l+1|n})$ associated to the filtration of Heisenberg and the bigraded space $\Sigma_\delta(\R^{2l+1|n})$ of fine symbols. In the last section, we show that those spaces are $\spo(2l+2|n)$-modules, where $\spo(2l+2|n)$ is the Lie sub-superalgebra of $\mathcal{K}(2l+1|n)$ constituted by the contact vector fields $X_f$ whose the superfunctions $f$ are degrees to most equal two. In order to facilitate the computations, we use the method used by C.H. Conley and V. Ovsienko in \cite{CoOv12}: we represent the symbols by the polynomials and the Lie derivatives by the differential operators. In this way, we see that the spaces $\mathcal{S}^k_\delta(\R^{2l+1|n})$ and $\mathcal{P}^d_\delta(\R^{2l+1|n})$ and $\Sigma^{k,d}_\delta(\R^{2l+1|n})$ are isomorphic to some sub-spaces of $$\mathcal{F}_\delta(\R^{2l+1|n})\otimes\mathrm{Pol}(T^*\R^{2l+1|n}).$$
We use the symbolical notation, i.e. the moments $\zeta,\alpha_i,\beta_i$ and $\gamma_i$ associated to the vector fields
$
 \partial_z,\quad A_i=\partial_{x_i}+y_i\partial_z,\quad B_i=-\partial_{y_i}+x_i\partial_z \quad\mbox{and}\; \bar{D}_i=\partial_{\theta_i}-\theta_i\partial_z 
$
respectively and compute the explicit formulas of the actions of $\spo(2l+2|n)\subset \mathcal{K}(2l+1|n)$ on the spaces $\mathcal{S}^k_\delta(\R^{2l+1|n})$ and $\mathcal{P}^d_\delta(\R^{2l+1|n})$ and $\Sigma^{k,d}_\delta(\R^{2l+1|n})$.

\section{Contact structure on $\R^{2l+1|n}$}
We consider the supermanifold $\R^{2l+1|n}$, where $l$ and $n$ are integers.
The standard contact structure on the supermanifold $\R^{2l+1|n}$ is defined by the kernel of the differential $1$-superforms $\alpha$ on $\R^{2l+1|n}$ which, in the system of Darboux coordinates $(z,x_i,y_i,\theta_j),\quad i=1,\cdots,l$ and $j=1,\cdots,n$ it can be written as 
\begin{equation}\label{STANDARDCONTACT}
\alpha=dz+\sum_{i=1}^l{(x_idy_i-y_idx_i)}+\sum_{i=1}^n{\theta_id\theta_i}.
\end{equation}
This differential $1$-superform $\alpha$ is called contact form on $\R^{2l+1|n}$ and we denote by $\mathrm{Tan}(\R^{2l+1|n})$ the space constituted of the elements $T_r,\quad 1\leqslant r\leqslant 2l+n$ of the kernel of $\alpha$. 

If we denote $q^A=(z,q^r)$ the generalized coordinate where 
\begin{equation}\label{gencoordin}
q^A=\left\lbrace 
\begin{array}{lcl}
 z\quad\mbox{if}\quad A=0\\
 x_A \quad\mbox{if}\quad 1\leqslant A\leqslant l,\\ 
 y_{A-l} \quad\mbox{if} \quad l+1\leqslant A\leqslant 2l\\ 
 \theta_{A-2l} \quad\mbox{if}\quad 2l+1\leqslant A\leqslant 2l+n 
\end{array}\right.
\end{equation}
we can write $\alpha$ in the following way
\[
\alpha=dz+\omega_{rs}q^rdq^s, \quad (\omega_{rs})=\left(
\begin{array}{cc|c}
0& \id_l&0\\
-\id_l&0&0\\
\hline
0&0&\id_n
\end{array}
\right).
\]
\begin{rmk}
We denote by $\omega^{sk}$ the elements of the matrix $(\omega^{sk})$ so that $(\omega_{rs})(\omega^{sk})=(\delta^k_r)$. We have thus
\[
(\omega^{rs})=\left(
\begin{array}{cc|c}
0&- \id_l&0\\
\id_l&0&0\\
\hline
0&0&\id_n
\end{array}
\right).
\] and $(\omega^{rs})=-(-1)^{\tilde{r}\tilde{s}}(\omega^{sr}).$
\end{rmk}
\begin{df}
We call the field of Reeb  on $\R^{2l+1|n}$, the vector field $T_0\in\mathrm{Vect}(\R^{2l+1|n})$ which, in the system of Darboux coordinates, one write 
$T_0=\partial_z$.
\end{df}
We can show that the field of Reeb is the unique vector field on $\R^{2l+1|n}$ so that $i(T_0)\alpha=1$ and $i(T_0)d\alpha=0$.\\
As proved in \cite{Nib15}, the vector fields $T_1,\cdots,T_{2l+n}\in \mathrm{Tan}\R^{2l+1|n}$, (i.e. the kernel of $\alpha$) are written explicitly as follow:
\begin{equation}\label{tangentdistr}
T_r=\left\lbrace 
\begin{array}{lcl} 
 A_r &:=& \partial_{x_r}+y_r\partial_z\quad\mbox{if}\quad 1\leqslant r\leqslant l\\ 
 -B_{r-l} &:=& \partial_{y_{r-l}}-x_{r-l}\partial_z\quad\mbox{if}\quad l+1\leqslant r\leqslant 2l\\ 
 \overline{D}_{r-2l}&:=& \partial_{\theta_{r-2l}}-\theta_{r-2l}\partial_z \quad\mbox{if}\quad 2l+1\leqslant r\leqslant 2l+n
\end{array}\right.
\end{equation}
If we write the vector field $T_r$ as follow $ T_r=\partial_{q^r}-\omega_{kr}q^k\partial_z$, then the following formulas are immediate.
\begin{equation}\label{TANGDIST}
T_r(q^k)=\delta_r^k,\quad T_r(z)=-\omega_{kr}q^k,\quad [T_r,T_j]=-2\omega_{rj}\partial_z,\quad T_r(z^2)=-2z\omega_{kr}q^k.
\end{equation}

\section{Module of densities on $\R^{2l+1|n}$}
In this section, we consider the spaces $C^\infty(\R^{2l+1|n})$ of superfunctions and the space $\mathrm{Vect}(\R^{2l+1|n})$ of all homogeneous vector fields on $\R^{2l+1|n}$.\\
We define an action $\mathbb{L}_X^\lambda$ of an element $X=\sum_AX^A\partial_{q^A}$ of $\mathrm{Vect}(\R^{2l+1|n})$ on the space of superfunctions $C^\infty(\R^{2l+1|n})$ as follow
\begin{equation}\label{Action}
\mathbb{L}_X^\lambda(g):=X(g)+\lambda \mathrm{div}(X)g,\forall g\in C^\infty(\R^{2l+1|n}),
\end{equation}
where $\mathrm{div}(X)=\sum_A(-1)^{\widetilde{X^A}\widetilde{q^A}}\partial_{q^A}X^A$.
We can show that for all $X,Y\in\mathrm{Vect}(\R^{2l+1|n}) $, we have
\[
[\mathbb{L}_X^\lambda,\mathbb{L}_Y^\lambda]=\mathbb{L}_{[X,Y]}^\lambda
\] 
\begin{df}
The module $\mathrm{Ber}_\lambda(\R^{2l+1|n})$ of tensor densities of weight $\lambda\in\R$ on $\R^{2l+1|n}$ is the space of superfunctions $C^\infty(\R^{2l+1|n})$ endowed with the action \eqref{Action} of $\mathrm{Vect}(\R^{2l+1|n})$. One can writes for $g\in C^\infty(\R^{2l+1|n})$, any $\lambda$-tensor density as $g|Dx|^\lambda$.
\end{df}
In particular case, when we consider a contact structure on $\R^{2l+1|n}$, we can define a sub space $\mathcal{K}(2l+1|n)$ of $\mathrm{Vect}(\R^{2l+1|n})$ constituted by the contact vector fields $X_f$ on $\R^{2l+1|n}$. Explicitly, we can show that the elements $X_f$ are expressed \cite{Nib15} by
\begin{equation}\label{CONTACTFORME}
X_f=f\partial_z-\frac{1}{2}(-1)^{\tilde{f}\tilde{T_r}}\omega^{rs}T_r(f)T_s,
\end{equation} where $\omega^{rs}$ denotes the elements of the matrix $(\omega^{rs})$  and $T_r$ denotes the elements of kernel of the standard contact form $\alpha$ on $\R^{2l+1|n}$ (see \eqref{STANDARDCONTACT} and \eqref{tangentdistr}).
In this way, we define an action of $X_f$ on $C^\infty(\R^{2l+1|n})$ by
\begin{equation}\label{action2}
L_{X_f}^\lambda(g)=X_f(g)+\lambda f'g,\quad \forall g\in C^\infty(\R^{2l+1|n})
\end{equation} and where $f'=\partial_zf$ and  where $(z,x^i,y^i,\theta^j),\quad i\in [1,l],\quad j\in[1,n]$ is the coordinates system of Darboux.
It is clear that
\[
[L_{X_f}^\lambda,L_{X_g}^\lambda]=L_{[X_f,X_g]}^\lambda=L_{X_{\{f,g\}}}^\lambda,
\]
where $\{f,g\}$ is the Lagrange bracket of the superfunctions $f$ and $g$.
\begin{df}
The module $\mathcal{F}_\lambda(\R^{2l+1|n})$ of $\lambda$-contact densities on $\R^{2l+1|n}$ is the space $C^\infty(\R^{2l+1|n})$ endowed with the action \eqref{action2} of $\mathcal{K}(2l+1|n)$. One can write any $\lambda$-contact density as $g\alpha^\lambda$ where $\alpha$ is a contact $1$-form on $\R^{2l+1|n}$ and for an arbitrary superfunction $g$.
\end{df}
We have the following result.
\begin{prop}\label{coucou}
If the superdimension $d=2l+1-n$ is different to $-1$, then the application 
\[
\varphi:\mathcal{F}_\lambda(M)\to\mathrm{Ber}_{\frac{2\lambda}{d+1}}(M):g\alpha^\lambda\mapsto g\vert Dx\vert^{\frac{2\lambda}{d+1}}
\] is an isomorphism of $\mathcal{K}(2l+1|n)$-modules.
\end{prop}
\begin{proof}
It is clear that $\varphi$ is bijective. We must show that the application $\varphi$ intertwines the action of $X_f$ on the spaces $\mathcal{F}_\lambda(\R^{2l+1|n})$ and $\mathrm{Ber}_{\frac{2\lambda}{d+1}}(\R^{2l+1|n})$, i.e.
\begin{equation}\label{IntertwAction}
\varphi(L^\lambda_{X_f}(g\alpha^\lambda))=\mathbb{L}^\frac{2\lambda}{d+1}_{X_f}(\varphi(g\alpha^\lambda)).
\end{equation}
Indeed, on the space $\mathcal{F}_\lambda(\R^{2l+1|n})$, the action of $\mathcal{K}(2l+1|n)$ is  given by 
\begin{equation}\label{LIEACTION}
L^\lambda_{X_f}(g\alpha^\lambda)= (X_f(g)+\lambda f'g)\alpha^\lambda=\left(X_f(g)+\lambda\frac{2\mathrm{Div}X_f}{d+1}g\right)\alpha^\lambda
\end{equation} where $\mathrm{div}(X_f)=\frac{2l+2-n}{2}f'$.
The second member of \eqref{IntertwAction} can be written as
\begin{equation}\label{FormulLieDeriv}
\left(X_f(g)+\frac{2\lambda}{d+1}\mathrm{div}(X_f)g\right)|Dx|^{\frac{2\lambda}{d+1}}.
\end{equation} 
If we apply the isomorphism $\varphi$ to the density given by \eqref{LIEACTION}, we obtain also \eqref{FormulLieDeriv}.
\end{proof}

In particular, if we denote by $T\R^{2l+1|n}$ the supertangent sheaf on $\R^{2l+1|n}$, we can define the application
\[
X. :\mathcal{F}_{-1}(\R^{2l+1|n})\to T\R^{2l+1|n}:f\alpha^{-1}\mapsto X_f
\] and this application establishes the intertwine of the representations $\mathcal{F}_{-1}(\R^{2l+1|n})$ and $T\R^{2l+1|n}$ of $\mathcal{K}(2l+1|n)$.

 If we denote by $\mathrm{Tan}\R^{2l+1|n}$ the space of all vector fields on $\R^{2l+1|n}$ which preserve the contact structure, we can also show that the space $T\R^{2l+1|n}$ is the direct sum of two spaces as follow
\[
T\R^{2l+1|n}=\mathrm{Tan}\R^{2l+1|n}\oplus\mathcal{K}(2l+1|n).
\]
The space $\mathrm{Tan}\R^{2l+1|n}$ is an $C^\infty(\R^{2l+1|n})$-module but is not a Lie superalgebra and the space $\mathcal{K}(2l+1|n)$ is not a module but it possess a structure of Lie superalgebra.

\section{Module of differential operators on $\R^{2l+1|n}$}

We begin here by the classical definitions of differential operators between spaces of densities.
\begin{df}
If we denote $(q^A)=(z,x^i,y^i,\theta^j)$, $(i\in [1,l],\quad j\in[1,n])$ the coordinates system of Darboux on $\R^{2l+1|n}$, we call differential operator $D$ of order $k$ on $\R^{2l+1|n}$, the application which maps on $\mathcal{F}_\lambda(\R^{2l+1|n})$ to $\mathcal{F}_\mu(\R^{2l+1|n})$ and it can be written in coordinates by
\[D:f\alpha^{\lambda}\mapsto(\sum_{I:|I|\leq k}{D_I\partial_{q^I}f)\alpha^{\mu}  
},\]
where $I=(i_0,i_1,\ldots,i_{2l+n})$ is a multi-index, $|I|=i_{0}+\ldots+i_{2l+n}$
 and $D_I$ is a superfunction for all $I$.
\end{df}

More explicitly, a differential operator $D$ can be written in coordinates by
 \begin{equation}\label{opdiff}
 \sum_{I:|I|\leq k}{D_I(\partial_{z})^{i_0}(\partial_{x_1})^{i_1}\ldots(\partial_{x_l})^{i_l}(\partial_{y_1})^{i_{l+1}}\cdots(\partial_{y_l})^{i_{2l}}(\partial_{\theta_1})^{i_{2l+1}}\ldots(\partial_{\theta_n})^{i_{2l+n}}}.
 \end{equation}
Because $\partial_{\theta_i}^2=0$, the exponents $i_{2l+1},\ldots,i_{2l+n}$ in the expression \eqref{opdiff} are at most equal to $1$.

The differential operators of order $0$ are simply the multiplication by $(\mu-\lambda)$-densities.
We define the space of differential operators as follow
  \[
\mathcal{D}_{\lambda\mu}(\R^{2l+1|n})=\bigcup_{k=0}^\infty\mathcal{D}_{\lambda\mu}^k(\R^{2l+1|n}).
\]
If $D\in \mathcal{D}_{\lambda\mu}(\R^{2l+1|n})$ and if $X\in\mathcal{K}(2l+1|n)$, then the action of $X$ on $\mathcal{D}_{\lambda\mu}(\R^{2l+1|n})$ is defined by the Lie derivative $\mathcal{L}_X^{\lambda\mu}$ via the following supercommutator  

   \begin{equation}\label{optoraction}
 \mathcal{L}_X^{\lambda\mu}D=L_X^{\mu}\circ D-(-1)^{\tilde{X}\tilde{D}}D\circ L_X^{\lambda}
 \end{equation} where $L_X^\lambda$ and $L_X^\mu$ are defined in \eqref{action2}.
  We have thus a  $\mathcal{K}(2l+1|n)$-module structure on the space $\mathcal{D}_{\lambda\mu}(\R^{2l+1|n})$. 
  
  \begin{rmk}
   The construction given above can be done with the representations $\mathrm{Ber}_\lambda(\R^{2l+1|n})$ and $\mathrm{Ber}_\mu(\R^{2l+1|n})$ of $\mathrm{Vect}(\R^{2l+1|n})$. We obtain the representation of $\mathrm{Vect}(\R^{2l+1|n})$ on the space of differential operators and it induces the representation of $\mathcal{K}(2l+1|n)$. The proposition \ref{coucou} allows us to show that these representations of $\mathcal{K}(2l+1|n)$ are isomorphic, modulo a change of an adequate weight.
  \end{rmk}

\section{Canonical filtration of $\mathcal{D}_{\lambda\mu}(\R^{2l+1|n})$}
Firstly, we show that the space $\mathcal{D}_{\lambda\mu}(\R^{2l+1|n})$ has the canonical filtration. Secondly, we show that the action of $\mathcal{K}(2l+1|n)$ on 
$\mathcal{D}_{\lambda\mu}(\R^{2l+1|n})$ induces an action on the associated graded space $\mathcal{S}_\delta(\R^{2l+1|n})$ that we call the space of symbols on $\R^{2l+1|n}$.
The subspaces $\mathcal{D}^l_{\lambda\mu}(\R^{2l+1|n})$ are stable under the action of $\mathcal{K}(2l+1|n)$.
\begin{prop}
If $X\in \mathcal{D}^k_{\lambda\mu}(\R^{2l+1|n})$ and if $X\in\mathcal{K}(2l+1|n)$ then $\mathcal{L}_X^{\lambda\mu}(D)\in\mathcal{D}^k_{\lambda\mu}(\R^{2l+1|n})$.
\end{prop}
\begin{proof}
We compute the terms of order $k+1$ in the expression of $\mathcal{L}_X^{\lambda\mu}(D)$. Because of definition of the action of $X$, we must evaluate the terms of order $k+1$ of $L^\mu_X\circ D$ and $D\circ L_X^\lambda$. We remark that these terms are identical. Therefore, $\mathcal{L}_X^{\lambda\mu}(D)$ is a  differential operator of order $k$.
\end{proof}
It is easy to see that the following inclusions are immediate
\[
\mathcal{D}_{\lambda\mu}^0(\R^{2l+1|n}) \subset\mathcal{D}_{\lambda\mu}^1(\R^{2l+1|n})\subset\mathcal{D}_{\lambda\mu}^2(\R^{2l+1|n})\subset\cdots\subset
\mathcal{D}_{\lambda\mu}^{l-1}(\R^{2l+1|n})\subset\mathcal{D}_{\lambda\mu}^l(\R^{2l+1|n})\subset\cdots,
 \]
and therefore we deduce that the spaces $\mathcal{D}_{\lambda\mu}^l(\R^{2l+1|n})$ define a filtration of the module of differential operators $\mathcal{D}_{\lambda\mu}(\R^{2l+1|n})$.
\begin{rmk}
If $\lambda=\mu$, the space $\mathcal{D}_{\lambda\lambda}(\R^{2l+1|n})$ is an associative and filtered superalgebra by composition of differential operators. It is therefore a filtration of algebra, i.e. 
\[
  \mathcal{D}^k_{\lambda\lambda}(\R^{2l+1|n}).\mathcal{D}^l_{\lambda\lambda}(\R^{2l+1|n})\subseteq\mathcal{D}^{k+l}_{\lambda\lambda}(\R^{2l+1|n}).
  \]
\end{rmk}

We can now define the graded space which is associated to this filtration of $\mathcal{D}^k_{\lambda\lambda}(\R^{2l+1|n})$.
\begin{df}
The space $\mathcal{S}_\delta(\R^{2l+1|n})$ is called space of principal symbols of order $k$ and defined by 
 \[\mathcal{S}^k_\delta(\R^{2l+1|n}):=\D^k_{\lambda\mu}(\R^{2l+1|n})/\D^{k-1}_{\lambda\mu}(\R^{2l+1|n}),\quad \delta=\mu-\lambda.\]
We define on $\D^k_{\lambda\mu}(\R^{2l+1|n})$ the following surjective application $\sigma_k$ 
\begin{equation}
  \sigma_k:\D^k_{\lambda\mu}(\R^{2l+1|n})\to\mathcal{S}^k_\delta(\R^{2l+1|n}):D\mapsto [D],
\end{equation} 
where $[D]$ means the equivalence class of $D$.
\end{df}
The action $L_X^\delta$ of $\mathcal{K}(2l+1|n)$ on $\mathcal{S}^k_\delta(\R^{2l+1|n})$ is induced by the action of $\mathcal{K}(2l+1|n)$ on $\D^k_{\lambda\mu}(\R^{2l+1|n})$, i.e. if $S=[D]$ where $D$ is given by the formulae \eqref{opdiff}, then
we have
 \[
 L^\delta_X(S):= [\mathcal{L}_X^{\lambda\mu}(D)].
 \]
 \begin{rmk}
 
 The previous constructions are also valid for the representations of $\mathrm{Vect}(\R^{2l+1|n})$ if we consider the spaces $\mathrm{Ber}_\lambda(\R^{2l+1|n})$ instead of the spaces $\mathcal{F}_\lambda(\R^{2l+1|n})$. The induced representations of $\mathcal{K}(2l+1|n)$ are isomorphic to those presented here, modulo a change of an adequate weight.
 \end{rmk}

\section{Filtration of Heisenberg of $\mathcal{D}_{\lambda\mu}(\R^{2l+1|n})$}
If the superspace $\R^{2l+1|n}$ is endowed with the standard contact form 
$\alpha$ then the superderivations on $C^{\infty}(\R^{2l+1|n})$ are generated by the field of Reeb $\partial_z$ and by the vector fields $T_1,\cdots,T_{2l+n}\in \mathrm{Tan}\R^{2l+1|n}$.
We can therefore define on the space $\D_{\lambda\mu}(\R^{2l+1|n})$ of differential operators another filtration and naturally another type of symbols.

\begin{prop}
If $K=(i_1,\ldots,i_{2l+n})$ is a multi-index of length $|K|=i_1+i_2+\ldots+i_{2l+n}$ and if $D$ is a differential operator of order $k$, then $D$ can be written in the follow unique form:
\begin{equation}\label{NORMALOPDIFF}
\sum_{K:c+|K|\leqslant k}D_{cK}\partial_{z}^{c}T^K,
\end{equation}
where $D_{cK}$ is a superfunction and $T^K=T^{i_1}_1\cdots T^{i_{2l+n}}_{2l+n}$.
\end{prop}
\begin{proof}
The existence is proven by the decomposition \[
T\R^{2l+1|n}=\mathrm{Tan}\R^{2l+1|n}\oplus\mathcal{K}(2l+1|n).
\] and because of the explicit form of vector fields $T_i$ we prove that the form is unique.
\end{proof}
\begin{df}
If $D$ is a differential operator of the form \eqref{NORMALOPDIFF}, then we say that $D$ has an order of Heisenberg equal to $d$ if $c+\frac{1}{2}|K|\leqslant d$ for all $c,K$. We denote by $\mathcal{H}^d_{\lambda\mu}(\R^{2l+1|n})$ the space of differential operators of order of Heisenberg equal to $d$ on $\R^{2l+1|n}$.
\end{df}
We can see that this space of differential operators is therefore filtered by the order of Heisenberg. Indeed, the total space $\mathcal{H}_{\lambda\mu}(\R^{2l+1|n})$ is the union of the spaces $\mathcal{H}^d_{\lambda\mu}(\R^{2l+1|n})$, i.e.
 \[
\mathcal{H}_{\lambda\mu}(\R^{2l+1|n})=\bigcup_{d\in\frac{1}{2}\N}\mathcal{H}^d_{\lambda\mu}(\R^{2l+1|n}). 
 \]
 Because $\mathcal{H}^d_{\lambda\mu}(\R^{2l+1|n})\subset\mathcal{H}^{d+\frac{1}{2}}_{\lambda\mu}(\R^{2l+1|n})$, we have the following inclusions 

\[
 \mathcal{H}^0_{\lambda\mu}(\R^{2l+1|n})\subset\mathcal{H}^{\frac{1}{2}}_{\lambda\mu}(\R^{2l+1|n})\subset\mathcal{H}^1_{\lambda\mu}(\R^{2l+1|n})\subset\cdots\subset\mathcal{H}^d_{\lambda\mu}(\R^{2l+1|n})\subset\mathcal{H}^{d+\frac{1}{2}}_{\lambda\mu}(\R^{2l+1|n})\subset\cdots
 \] for all $d\in\frac{1}{2}\N$.

\begin{df}
The graded space associated to the space $\mathcal{H}^d_{\lambda\mu}(\R^{2l+1|n})$
is denoted by $\mathcal{P}_\delta(\R^{2l+1|n})$. We have therefore
\begin{equation}\label{PD}
\mathcal{P}_{\delta}(\R^{2l+1|n}):=\bigoplus_{d\in\frac{1}{2}\N} \mathcal{P}_{\delta}^{d}(\R^{2l+1|n}):=\sum_{d\in\frac{1}{2}\N} \mathcal{H}_{\lambda\mu}^{d}(\R^{2l+1|n})/\mathcal{H}_{\lambda\mu}^{d-\frac{1}{2}}(\R^{2l+1|n}),
\end{equation} where $\delta=\mu-\lambda$.
\end{df}

The canonical projection define the application $h\sigma$ called Heisenberg symbol map as follow
\[{\rm h}\sigma:\mathcal{H}_{\lambda\mu}^{d}(M)\to\mathcal{P}_{\delta}^{d}(M):D\mapsto [D],\]
where $[D]$ means the equivalence class of $D$ in the quotient $\mathcal{H}_{\lambda\mu}^{d}(\R^{2l+1|n})/\mathcal{H}_{\lambda\mu}^{d-\frac{1}{2}}(\R^{2l+1|n})$.
\section{Bifiltration of $\mathcal{D}_{\lambda\mu}(\R^{2l+1|n})$ and the associated bigraded space $\Sigma_\delta(\R^{2l+1|n})$ }

In this section, we show that the canonical filtration and the filtration of Heisenberg induce a particular filtration called bifiltration of $\mathcal{D}_{\lambda\mu}(\R^{2l+1|n})$. We generalize in super case, the model used by C.Conley and V.Ovsienko in \cite{CoOv12} in the even case.

\begin{df}
We define a bifiltration on $\mathcal{D}_{\lambda\mu}(\R^{2l+1|n})$ by
\[
\mathcal{D}_{\lambda\mu}^{k,d}(\R^{2l+1|n}):=\mathcal{D}_{\lambda\mu}^{k}(\R^{2l+1|n})\cap\mathcal{H}_{\lambda\mu}^{d}(\R^{2l+1|n}). 
\] 
The bigraded space $\Sigma_\delta(\R^{2l+1|n})$ associated to the bifiltration $\mathcal{D}_{\lambda\mu}^{k,d}(\R^{2l+1|n})$ is defined by
\begin{equation}\label{BIGRADED}
\Sigma_{\delta}(\R^{2l+1|n})=\bigoplus_{k=0}^\infty\bigoplus_{d\in\frac{1}{2}\N}\Sigma_{\delta}^{k,d}(\R^{2l+1|n})=\bigoplus_{k=0}^\infty\bigoplus_{d\in\frac{1}{2}\N} \mathcal{D}_{\lambda\mu}^{k,d}(\R^{2l+1|n})/(\mathcal{D}_{\lambda\mu}^{k-1,d}(\R^{2l+1|n})+\mathcal{D}_{\lambda\mu}^{k,d-\frac{1}{2}}(\R^{2l+1|n})).
\end{equation}
The elements of $\Sigma_{\delta}(\R^{2l+1|n})$ are called fine symbols.
\end{df}
We define accordingly the fine symbol map by
\[{\rm f}\sigma_{k,d}:\mathcal{D}_{\lambda\mu}^{k,d}(\R^{2l+1|n})\to\Sigma_{\delta}^{k,d}(\R^{2l+1|n}):D\mapsto [D],\] where the bracket means the equivalence class of $D$
in $\mathcal{D}_{\lambda\mu}^{k,d}(\R^{2l+1|n})/(\mathcal{D}_{\lambda\mu}^{k-1,d}(\R^{2l+1|n})+\mathcal{D}_{\lambda\mu}^{k,d-\frac{1}{2}}(\R^{2l+1|n}))$.
To justify the terminology of fine symbol, we refer to \cite{CoOv12}.
\begin{rmk}
By definition of action, the applications ${\rm f}\sigma_{k,d}$ and $\sigma_k$ are $\mathcal{K}(2l+1|n)$-equivariant, i.e.
\[
L_{X_f}^{\delta,\mathcal{S}}\circ\sigma_k=\sigma_k\circ\mathcal{L}_{X_f}^{\lambda\mu}\quad\mbox{on}\;\;\mathcal{D}_{\lambda\mu}^k(\R^{2l+1|n}),
\] and
\[
L_{X_f}^{\delta,\Sigma}\circ {\rm f}\sigma_{k,d}={\rm f}\sigma_{k,d}\circ \mathcal{L}_{X_f}^{\lambda\mu}\quad\mbox{on}\;\; \mathcal{D}^{k,d}_{\lambda\mu}(\R^{2l+1|n}),
\] where $L_{X_f}^{\delta,\mathcal{S}}$ and $L_{X_f}^{\delta,\Sigma}$ denote respectively the action $L_{X_f}^{\delta}$ on $\mathcal{S}(\R^{2l+1|n})$ and $L_{X_f}^{\delta}$ on $\Sigma(\R^{2l+1|n})$.
\end{rmk}
We have thus the following main result
\begin{prop}
The action of $\mathcal{K}(2l+1|n)$ preserve the filtrations of $\D_{\lambda\mu}(\R^{2l+1|n})$, $\mathcal{H}_{\lambda\mu}(\R^{2l+1|n})$ and $ \mathcal{D}^{k,d}_{\lambda\mu}(\R^{2l+1|n})$.
\end{prop}
\begin{proof}
We must verify  that the form of differential operator defined by the formulae \eqref{NORMALOPDIFF} doesn't be modified by the action of $\mathcal{K}(2l+1|n)$.
Therefore, we consider a differential operator $D$ of order $k$ and we compute $\mathcal{L}_{X_f}^{\lambda\mu}(D)$ as follow
\begin{eqnarray*}
\mathcal{L}^{\lambda\mu}_{X_f}(D)&=&(X_f+\mu f')D-(-1)^{\tilde{f}\tilde{D}}D.(X_f+\lambda f')\\
                        &=&[X_f,D]+D_k,
\end{eqnarray*}
where $D_k= \mu f'D-(-1)^{\tilde{f}\tilde{D}}\lambda D.f'$.
We can see that the term $D_k$ is a differential operator of the same order than $D$. Because of 
\[
[X_f,T_I]\in<T_1,\cdots,T_{2l+n}>,
\] the term $[X_f,D]$ is a differential operator of order $k$. Therefore the filtrations of $\D_{\lambda\mu}(\R^{2l+1|n})$, $\mathcal{H}_{\lambda\mu}(\R^{2l+1|n})$ and $ \mathcal{D}^{k,d}_{\lambda\mu}(\R^{2l+1|n})$ are preserved by the canonical action of $X_f$.
\end{proof}

The action of $\mathcal{K}(2l+1|n)$ on the space $\mathcal{H}_{\lambda\mu}(\R^{2l+1|n})$ induces the structure of $\mathcal{K}(2l+1|n)$-module on $\mathcal{P}^d_\delta(\R^{2l+1|n})$. If $L_{X_f}^\mathcal{H}$ denotes the action of $X_f$ on $\mathcal{P}^d_\delta(\R^{2l+1|n})$ and if $[D]\in \mathcal{P}^d_\delta(\R^{2l+1|n})$, then we can see that 
\[
 L^{\mathcal{H}}_{X_f}[D]=[\mathcal{L}^{\lambda\mu}_{X_f} D].
 \]
\section{ $\spo(2l+2|n)$-modules on the spaces $\mathcal{S}_\delta$,$\mathcal{P}_\delta$ and $\Sigma_\delta$}
In this section, we consider the Lie sub-superalgebra $\spo(2l+2|n)$ of $\mathcal{K}(2l+1|n)$ constituted by the contact vector fields $X_f$ whose the superfunctions $f$ are degrees to most equal two.
We show that the $\spo(2l+2|n)$-modules on the spaces of differential operators $\D_{\lambda\mu}(\R^{2l+1|n})$, $\mathcal{H}_{\lambda\mu}(\R^{2l+1|n})$ and $ \mathcal{D}^{k,d}_{\lambda\mu}(\R^{2l+1|n})$ are induced on the symbols spaces $\mathcal{S}_\delta(\R^{2l+1|n})$,$\mathcal{P}_\delta(\R^{2l+1|n})$ and $\Sigma_\delta(\R^{2l+1|n})$. We give also the explicit formulas of those actions: to facilitate the computations, we use the symbolical notations by defining the isomorphisms between the spaces $\mathcal{S}_\delta(\R^{2l+1|n})$ and $\mathcal{P}_\delta(\R^{2l+1|n})$ and some subspaces of the space $\mathcal{F}_{\delta}\otimes\mathrm{Pol}(T^*\R^{2l+1|n})$ (see also \cite{CoOv12} in the even case). Therefore we deduce the form of the elements of $\Sigma_\delta^{k,d}(\R^{2l+1|n})$.\\

First, we denote by $\xi_0$ the moment associated to the vector fields $\partial_z$ and by $\xi_r$ the moment associated to the vector fields $T_r$. We obtain therefore the following isomorphism
\begin{equation}\label{POLYNOTATION}
\mathcal{S}_{\delta}^{k}(\R^{2l+1|n})\cong\langle\alpha^{\delta} \xi_{0}^{c}\xi^{I}, c+|I|=k\rangle,
\end{equation}
where $I=(i_1,\ldots,i_{2l+n})$ is a multi-index of length $|I|=i_1+\ldots+i_{2l+n}$ and $\xi^{I}=\xi_{1}^{i_1}\ldots\xi_{2l+n}^{i_{2l+n}}$.\\

If we denote by $\varphi$ the isomorphism defined by the formulae \eqref{POLYNOTATION} and  because of the form \eqref{NORMALOPDIFF}, we obtain 
\begin{equation}\label{ISOMOR_POLYN}
\varphi:\sum_{\substack{c,I\\c+|I|\leqslant k}}D_I\partial_{z}^{c}T_1^{i_1}\ldots T_{2l+n}^{i_{2l+n}}+\mathcal{D}_{\lambda,\mu}^{k-1}(\R^{2l+1|n})\mapsto\sum_{\substack{c,I\\c+|I|=k}}D_I\alpha^{\delta}\xi_0^{c}\xi_1^{i_1}\ldots\xi_{2l+n}^{i_{2l+n}}.
\end{equation}
From now we represent $\xi_0$ by $\zeta$ and the notation $\xi_i$ means the unified notation of the moments $\alpha_i,\beta_i$ and $\gamma_i$. The notation $T_i$ (see the formulae \eqref{tangentdistr}) means therefore the unified notation of vector fields  
 \begin{equation}\label{TANGDISTRIB}
 A_i=\partial_{x_i}+y_i\partial_z,\quad B_i=-\partial_{y_i}+x_i\partial_z,\quad \bar{D}_i=\partial_{\theta_i}-\theta_i\partial_z .
\end{equation}
If the multi-indexes $I,J,T$ are respectively given by $I=(i_1,i_2,\cdots,i_l)$, $J=(j_1,j_2,\cdots,j_l)$ and $T=(t_1,t_2,\cdots,t_n)$ then the quantities $A_1^{i_1}\cdots A_l^{i_l}, \quad B_1^{j_1}\cdots B_l^{j_l}$ and $\bar{D}_1^{t_1}\cdots \bar{D}_n^{t_n} $ are represented by $A^I, B^J$ and $\bar{D}^T$. We represent by $K$ the multi-index $(I,J,T)$ and by $|K|$, $|I|$, $|J|$ and $|T|$
their length respectively.\\
The space $\mathcal{P}_{\delta}^{d}(\R^{2l+1|n})$ defined by the formula \eqref{PD} is isomorphic to the subspace of $\mathcal{F}_{\delta}\otimes\mathrm{Pol}(T^*\R^{2l+1|n})$. More precisely, we have
\[\mathcal{P}_{\delta}^{d}(\R^{2l+1|n})\cong\langle D_{c,K}\alpha^\delta\zeta^{c}\alpha^{I}\beta^{J}\gamma^T,\quad c+\frac{1}{2}|K|=d\rangle,\] for all $c,I,J,K$ and this isomorphism is explicitly given by
\begin{equation}\label{ISOMOR_POLYN0}
\varphi:\sum_{\substack{c,K\\c+\frac{1}{2}|K|\leqslant d}}{D_{c,K}\partial_{z}^{c}A^IB^J\bar{D}^T}+\mathcal{H}_{\lambda,\mu}^{d-\frac{1}{2}}(\R^{2l+1|n})\mapsto\sum_{\substack{c,K\\c+\frac{1}{2}|K|=d}}{D_{c,K}\alpha^\delta\zeta^{c}\alpha^I\beta^J\gamma^T}.
\end{equation}
In the same way can see that the space $\Sigma_{\delta}^{k,d}(\R^{2l+1|n})$ is also isomorphic to the subspace of $\mathcal{F}_{\delta}\otimes\mathrm{Pol}(T^*\R^{2l+1|n})$ and we have
\[
\Sigma_{\delta}^{k,d}(\R^{2l+1|n})\cong\langle D_{c,K}\alpha^\delta\zeta^c \alpha^I\beta^{J}\gamma^T,\quad c+\frac{1}{2}|K|=d,\quad c+|K|=k\rangle
\]
and more explicitly, we have
\begin{multline*}
\varphi:\sum_{\substack{c,J\\c+\frac{1}{2}|K|\leqslant d\\c+|K|\leqslant k}}D_{c,K}{\partial_{z}^{c}A^IB^J\bar{D}^T}+(\mathcal{D}_{\lambda,\mu}^{k,d-\frac{1}{2}}(\R^{2l+1|n})+\mathcal{D}_{\lambda,\mu}^{k-1,d}(\R^{2l+1|n}))\mapsto\\
\sum_{\substack{c,K\\c+\frac{1}{2}|K|=d\\c+|K|=k}}D_{c,K}{\alpha^\delta\zeta^{c}\alpha^I\beta^J\gamma^T}.
\end{multline*}

It is now possible to write the space $\mathcal{P}_\delta(\R^{2l+1|n})$ as a direct sum of the spaces $\Sigma_\delta^{k,d}(\R^{2l+1|n})$ as follow:

\begin{equation}\label{FINSYMBOL_HESEIN}
\mathcal{P}_\delta(\R^{2l+1|n})=\bigoplus_{d\in\frac{1}{2}\N}{\mathcal{H}^d(\R^{2l+1|n})/
\mathcal{H}^{d-\frac{1}{2}}(\R^{2l+1|n})}
=\bigoplus_{d\in\frac{1}{2}\N}\bigoplus_{k=\lceil d\rceil}^{2d}\Sigma_\delta^{k,d}(\R^{2l+1|n})
\end{equation}
where $\lceil x\rceil:=\mathrm{inf}\{n\in\N:n\geqslant x\}$. \\
We recall that if $D\in\mathcal{H}_{\lambda,\mu}^{d}(\R^{2l+1|n})$ and if $X_f$ is the contact vector fields, then the Lie derivative in the direction of $X_f$, denoted by $L_{X_{f}}^{\mathcal{H}}D$ is given by
\[(L_{X_{f}}^{\mathcal{H}}D)(\psi):=L_{X_{f}}(D\psi)-(-1)^{\tilde{f}\tilde{D}}D(L_{X_{f}}\psi),\] and it is the element of $\mathcal{H}_{\lambda,\mu}^{d}(\R^{2l+1|n})$ thanks to the fact that the action of $X_f$ preserves $\mathrm{Tan}\R^{2l+1|n}$.
Therefore, the canonical action of $X_f$ preserves also the bifiltered space $\mathcal{D}_{\lambda,\mu}^{k,d}(\R^{2l+1|n})$.\\

If we consider the Lie sub-superalgebra $\spo(2l+2|n)$ of $\mathcal{K}(2l+1|n)$ constituted by the contact vector fields $X_f$ whose the superfunctions $f$ are degrees to most equal two, we  can see that the action of contact vector fields $X_f$ on $\mathcal{H}^d_{\lambda,\mu}(\R^{2l+1|n})$ and on $\mathcal{D}_{\lambda,\mu}^{k,d}(\R^{2l+1|n})$ induces the $\spo(2l+2|n)$-actions on the spaces $\mathcal{P}^d_{\delta}(\R^{2l+1|n})$ and $\Sigma_{\delta}^{k,d}(\R^{2l+1|n})$.

The following theorem gives the explicit formulas of the actions of the Lie superalgebra $\spo(2l+2|n)$ on the spaces 
$\mathcal{S}_{\delta}(\R^{2l+1|n})$,$\mathcal{P}_\delta(\R^{2l+1|n})$ and on the fine symbol space $\Sigma_\delta(\R^{2l+1|n})$.

\begin{theorem}
If $X_f\in \spo(2l+2|n)$ and if we denote by $L_{X_f}^{\mathcal{P}}$ (resp. $L_{X_f}^{\Sigma}$, resp. $L_{X_f}^{\mathcal{S}}$ ) the actions of $X_f$ on $\mathcal{P}_{\delta}(\R^{2l+1|n})$ (resp. $\Sigma_{\delta}(\R^{2l+1|n})$; resp. $\mathcal{S}_{\delta}(\R^{2l+1|n})$) then these actions are given by
\begin{enumerate}
\item[i)] \begin{multline}\label{derivecontact}
L^\Sigma_{X_f}=f\partial_z+\partial_z(f)\left(\delta-\mathcal{E}_\zeta\right)
-\frac{1}{2}(-1)^{\tilde{f}\tilde{T_r}}\omega^{rs}T_r(f)T_s\\
+\frac{1}{2}(-1)^{\tilde{f}(\tilde{T}_i+\tilde{T}_r)}\omega^{rs}T_iT_r(f)\xi_s
\partial_{\xi_i}.
\end{multline}
\item[ii)] $L_{X_f}^{\mathcal{P}} = L_{X_f}^{\Sigma}$, \\

\item[iii)] \begin{equation}\label{Normal}
L^S_{X_f}= L^\Sigma_{X_f}+\frac{1}{2}(-1)^{\tilde{f}\tilde{T}_r}\omega^{rs}T_r(f')\xi_s
\partial_{\zeta}, 
\end{equation}
\end{enumerate}
where the notation $\mathcal{E}_\zeta$ denotes the Euler operator $\zeta\partial_\zeta$ and the notations $\partial_z$ and $T_i$ denote the action of the vector fields $\partial_z$ and $T_i$ on the coefficients of the symbol.
\end{theorem}
\begin{proof}
The spaces $\Sigma(\R^{2l+1|n}):=\bigcup_{\delta\in\mathbb{R}}\Sigma_{\delta}(\R^{2l+1|n})$, $\mathcal{P}(\R^{2l+1|n}):=\bigcup_{\delta\in\mathbb{R}}\mathcal{P}_{\delta}(\R^{2l+1|n})$ and $\mathcal{S}(\R^{2l+1|n}):=\bigcup_{\delta\in\mathbb{R}}\mathcal{S}_{\delta}(\R^{2l+1|n})$ are actually algebras for the canonical product of symbols. We can consider the operators  $\tilde{L}_{X_f}^{\Sigma}$, $\tilde{L}_{X_f}^{\mathcal{P}}$ and $\tilde{L}_{X_f}^{\mathcal{S}}$  acting respectively on $\Sigma(\R^{2l+1|n})$, $\mathcal{P}(\R^{2l+1|n})$ and $\mathcal{S}(\R^{2l+1|n})$ whose the restrictions on the spaces $\Sigma_{\delta}(\R^{2l+1|n})$, $\mathcal{P}_{\delta}(\R^{2l+1|n})$ and $\mathcal{S}_{\delta}(\R^{2l+1|n})$ are given by $L_{X_f}^{\Sigma}$, $L_{X_f}^{\mathcal{P}}$ and $L_{X_f}^{\mathcal{S}}$ for all $\delta\in\mathbb{R}$. The operators $\tilde{L}_{X_f}^{\Sigma}$, $\tilde{L}_{X_f}^{\mathcal{P}}$ and $\tilde{L}_{X_f}^{\mathcal{S}}$ are actually derivations of the spaces $\Sigma(\R^{2l+1|n})$, $\mathcal{P}(\R^{2l+1|n})$ and $\mathcal{S}(\R^{2l+1|n})$.
We can compute the actions $L_{X_f}^\Sigma$ and $L_{X_f}^\mathcal{S}$ on the generators $\zeta$, $\xi_i$ and $g\alpha^\delta$ of the spaces $\Sigma_\delta(M)$ and $\mathcal{S}_\delta(M)$.\\
If the application of $L_{X_f}^\Sigma$ and $L_{X_f}^\mathcal{S}$ on those generators coincide with the actions of the second members of equations \eqref{derivecontact} and \eqref{Normal} on the same generators and thanks that the right members of those equations are the derivation operators, then the equations \eqref{derivecontact} and \eqref{Normal} are verified.\\
Thanks to the isomorphism $\varphi$ defined by \eqref{ISOMOR_POLYN0}, we compute the operators $L_{X_f}^\Sigma$ and $L_{X_f}^\mathcal{S}$ on the generators $\zeta, \xi_i$ and $g\alpha^\delta$ of the spaces $\Sigma_\delta(M)$ and $\mathcal{S}_\delta(M)$ by using  respectively the Lie derivative of differential operators $\partial_z,T_i$ and $g\alpha^\delta$.\\
We obtain 
\begin{eqnarray*}
L^\Sigma_{X_f}(T_i)&=&[X_f,T_i]+\delta f'T_i\\
               &=&[f\partial_z,T_i]+\delta f'T_i-\frac{1}{2}(-1)^{\tilde{f}\tilde{T}_r}\omega^{rs}\left(-(-1)^{\tilde{f}\tilde{T}_i}[T_i,T_r(f)T_s]\right).\\
               &=&-(-1)^{\tilde{f}\tilde{T_i}}T_i(f)\partial_z+\delta f'T_i\\
               &+&\frac{1}{2}(-1)^{(\tilde{T}_r+\tilde{T_i})\tilde{f}}\omega^{rs}\left(T_iT_r(f)T_s
              +(-1)^{\tilde{T}_i(\tilde{T}_r+\tilde{f})}T_r(f)[T_i,T_s]\right).
\end{eqnarray*}
Since the commutator $[T_i,T_s]$ is equal to $-2\omega_{is}\partial_z$ and using the isomorphism $\varphi$ given by \eqref{ISOMOR_POLYN0}, we obtain
\begin{equation}\label{Lie1}
L_{X_f}^\Sigma(\xi_i)=\left(\delta f'\Id+\frac{1}{2}(-1)^{\tilde{f}(\tilde{T}_r+\tilde{T}_i)}\omega^{rs}T_iT_r(f)\xi_s\partial_{\xi_i}\right)(\xi_i).
\end{equation}
If we compute $L_{X_f}^\Sigma(\partial_z)$ and using the isomorphism $\varphi$ given by \eqref{ISOMOR_POLYN0}, we obtain
\begin{equation}\label{Lie2}
L_{X_f}^\Sigma(\zeta)=\left(\delta-1\right)f'\Id(\zeta),
\end{equation}
and finally we obtain in the same conditions
\begin{equation}\label{Lie3}
L_{X_f}^\Sigma(g\alpha^\delta)=\left(f\partial_z+\delta f'\Id-(-1)^{\tilde{f}\tilde{T_r}}\frac{1}{2}\omega^{rs}T_r(f)T_s\right)(g\alpha^\delta).
\end{equation}

We can see that if we restrict the formula \eqref{derivecontact} to the generators $T_i$, $\partial_z$ and $g\alpha^\delta$, we obtain respectively the formulas (\ref{Lie1}), (\ref{Lie2}) et (\ref{Lie3}). The proof is also the same on the formula  \eqref{Normal}.
\end{proof}

\begin{rmk}
The spaces of symbols $\mathcal{S}_\delta(\R^{2l+1|n})$,$\mathcal{P}_\delta(\R^{2l+1|n})$ and $\Sigma_\delta(\R^{2l+1|n})$ are $\spo(2l+2|n)$-modules.
\end{rmk}

\section{Acknowledgments}
It is a pleasure to thank F. Radoux, P. Mathonet, and J.P. Michel for numerous fruitful discussions and for their interest in our work.

\end{document}